\documentclass{opt2024} 
\usepackage[asterism]{sectionbreak}
\usepackage{xcolor}    
\usepackage[inline, shortlabels]{enumitem}

\usepackage{bm}

\def\ve{{\bm{e}}}

\def\vp{{\bm{p}}}
\def\vq{{\bm{q}}}

\def\vu{{\bm{u}}}
\def\vv{{\bm{v}}}

\def\vx{{\bm{x}}}
\def\vy{{\bm{y}}}
\def\vz{{\bm{z}}}

\def\mI{{\bm{I}}}
\def\mA{{\bm{A}}}

\def\mM{{\bm{M}}}

\newcommand{\norm}[1]{\left\| #1 \right\|}
\def\R{{\mathcal{R}}}
\def\H{{\mathcal{H}}}
\def\T{{\operatorname{T}}}
\def\F{{\operatorname{F}}}
\def\CC{{\mathbb{C}}}

\title[On the Hypomonotone Class of Variational Inequalities]{On the Hypomonotone Class of Variational Inequalities}

\usepackage{xcolor}         
\definecolor{DarkGreen}{rgb}{0.1,0.5,0.1}
\definecolor{DarkRed}{rgb}{0.5,0.1,0.1}
\definecolor{DarkBlue}{rgb}{0.1,0.1,0.5}
\usepackage{hyperref}       
\hypersetup{
    unicode=false,          
    pdftoolbar=true,        
    pdfmenubar=true,        
    pdffitwindow=false,     
    pdfnewwindow=true,      
    colorlinks=true,       
    linkcolor=DarkBlue,          
    citecolor=DarkGreen,        
    filecolor=DarkRed,    
    urlcolor=DarkBlue,          
    pdftitle={},
    pdfauthor={},
}

\begin{document}
\maketitle \noindent

\begin{minipage}[t]{0.4\textwidth}
    \centering
    \textbf{Khaled Alomar}\\
    CISPA Helmholtz Center \\for Information Security \\
    \texttt{khaled.alomar@cispa.de}
\end{minipage}
\hfill
\begin{minipage}[t]{0.5\textwidth}
    \centering
    \textbf{Tatjana Chavdarova}\\
     Politecnico di Milano \&\\ University of California, Berkeley
    \texttt{tatjana.chavdarova@berkeley.edu}
\end{minipage}\vspace{5\baselineskip}

\begin{abstract}
This paper studies the behavior of the extragradient algorithm~\citep{korpelevich1976extragradient} when applied to hypomonotone operators, a class of problems that extends beyond the classical monotone setting. To support the understanding of this variational inequality problem class, we focus on a subclass of hypomonotone linear operators, characterizing them based on their eigenvalues and providing concrete examples. While the extragradient method is widely recognized for its efficiency in solving variational inequalities involving monotone and Lipschitz continuous operators, we demonstrate that it does not guarantee convergence in the hypomonotone case. In particular, we construct a counterexample where the extragradient method diverges regardless of the step size. A numerical experiment is presented to support this result.
\end{abstract}

\section{Introduction}\label{sec:intro}

We are interested in numerically solving the following \textit{Variational Inequality}~\citep[VI,][]{stampacchia_formes_1964,facchineiFiniteDimensionalVariationalInequalities2003} problem. Given an \emph{operator} $F\colon \R^n\to \R^n$, the goal is to:
\begin{equation} \label{eq:vi} \tag{VI}
  \text{\emph{find }}\ \vx^\star \qquad\text{s.t.}\qquad \langle \vx -\vx^\star, \F(\vx^\star)\rangle \geq 0, \quad \forall \vx \in \R^n \,.
\end{equation}
Special instances of \eqref{eq:vi} include standard minimization with $F\equiv  \nabla f, f\colon\R^n\to\R$, zero-sum min-max, and general sum multi-player games. 
This problem has gained recent popularity in machine learning, due to several instances that cannot be modeled by minimization only, such as Generative Adversarial Networks~\citep{goodfellow2014generative}, robust versions of classification~\citep[e.g.,][]{szegedy2014intriguing,madry2018towards,NEURIPS2020_02f657d5,causal2020}, actor-critic methods~\citep{konda1999}, and multi-agent reinforcement learning~\citep[e.g.,][]{bertsekas2021rollout,lowe2017multi}.

The extragradient algorithm, introduced by~\citet{korpelevich1976extragradient} in 1976, is a foundational iterative method used to solve variational inequality problems involving monotone operators. A detailed description of the algorithm is provided in \S~\ref{sec:preliminaries}.
Unlike gradient descent, the latest output (last iterate) of the extragradient method converges when applied to monotone and Lipschitz continuous operators; definitions can be found in \S~\ref{sec:preliminaries}. Its widespread use stems from its simplicity and efficiency for this class of problems.

However, in real-world applications, operators may not always exhibit monotonicity or cocoercivity. Instead, they may exhibit hypomonotonicity, a weaker form of the monotonicity condition. 
Hypomonotone operators enable the study of more general and complex scenarios, where finding solutions is more challenging but remains feasible for analysis.

Hypomonotonicity arises in a variety of contexts, including equilibrium problems, optimization with non-convex structures, and certain game-theoretic models. While monotone problems have been extensively studied, the class of hypomonotone problems remains relatively underexplored despite its practical significance; refer to \S~\ref{sec:related_works} for an overview. Understanding hypomonotone problems is crucial for assessing whether methods that guarantee convergence for monotone problems can solve more general problem classes.

\paragraph{Contributions.} Our contributions include:
\begin{itemize}
    \item We present two illustrative examples of problems that are hypomonotone but not monotone. Additionally, we characterize a simple class of problems exhibiting hypomonotonicity, based on their spectral properties. See \S~\ref{sec:hm_examples} for details. 
    \item We demonstrate that the seminal extragradient method fails to converge for general instances of hypomonotone \ref{eq:vi}s. Specifically, we construct a counterexample that is hypomonotone (but not monotone) and show that the extragradient method diverges for this instance regardless of the step size. Refer to \S~\ref{sec:eg_divergence}.
\end{itemize}

\section{Related Works}\label{sec:related_works}

\paragraph{Monotonicity.} 
Saddle point problems and VIs are widely studied, earliest works date to the 1960s~\citep{lions-stampacchia1967,lewy-stampacchia1969}.
For monotone VIs (Def.~\ref{def:monotone}),
~\citet{golowich2020last,golowich2020noregret} established that the lower bound of general $\tilde{p}$-\textit{stationary canonical linear iterative}  ($\tilde{p}$-SCLI) first-order methods~\citep{arjevani2016pscli} is  $\mathcal{O}(\frac{1}{ \tilde{p} \sqrt{K}})$, where $K$ is the iteration count; the result refers to the last iterate.
In her seminal work, \citet{korpelevich1976extragradient} proves the convergence of the iterates of the extragradient method for monotone  \ref{eq:vi}s. 
\citet{golowich2020last} obtained a last iterate rate for extragradient in terms of the gap function, relying on first- and second-order smoothness assumptions.
~\citet{gorbunov2022extra} and~\citet{gorbunov2022}  obtained a rate of $\mathcal{O}(\frac{1}{K})$ for extragradient~\citep{korpelevich1976extragradient} and \emph{optimistic gradient descent}~\citep{popov1980}, respectively---in terms of
reducing the squared norm of the operator, relying on the first-order operator smoothness.
~\citet{golowich2020last} and~\citet{chavdarova2023hrdes} provided the best iterate rate for optimistic gradient descent while assuming first-order smoothness of the operator $F$.
For monotone VIs,
\textit{(i)} with simple constraints,~\citet{cai2022constrVI} showed the last iterate convergence of extragradient and optimistic GDA, while
\textit{(ii)} \citet{yang2023acvi} showed the convergence for the last iterate of the \emph{ACVI} method proposed therein; the latter addressing the general constraints setting. Both results refer to the convergence of the gap function and match the lower bound.

\paragraph{Beyond monotonicity.}
Several works study the case where one of the two players is nonconvex. For instance,~\citet{nouiehed2019noncvx,lin20a,lin20b,kong2021,fiez2021global,ostrovskii2021,botAxel2023} focus on nonconvex-concave games, and~\citet{unfied2023} study convex-nonconcave problems.
~\citep{yang2020} consider problems that satisfy the two-sided Polyak-{\L}ojasiewicz inequality~\citep{polyak1963}.
\citet{diakonikolas21structured,bohm2022,pethick2022escaping} study a structured class of nonconvex-nonconcave min-max problems exhibiting so-called \emph{weak Minty} solutions~\citep{diakonikolas21structured}.

\paragraph{Beyond monotonicity: cohypomonotonicity (or $\rho$-comonotonicity).}
\citet{bauschke2021} characterize relations between nonexpansive maps with resolvents of cohypomonotone operators (referred as \textit{comonotonicity} therein); see definition in Appendix \ref{app:background}.
\citet{ATTOUCH20181095} discuss relations between cohypomonotonicity and other operator classes.
\citet{combettes2001proximal} show the weak convergence of an inexact, relaxed proximal point method~\citep{krasnoselskii1995ppm,Moreau1965,Martinet1970RegularisationDV} in the cohypomonotone setting. 
\citet{trandinh2023} studies accelerated versions of extragradient on cohypomonotone \ref{eq:vi}s with the additional assumed structure that the overall operator is a sum of two operators, with one Lipschitz continuous and the other possibly multivalued. 

\paragraph{Beyond monotonicity: hypomonotonicity.}
An operator $\F$ is hypomonotone if and only if its operator-inverse $\F^{-1}$ is cohypomonotone (see Definitions ~\ref{def:cohypomonotone} and ~\ref{def:invese-operator} in the Appendix).
\citet{iusem2003} study hypomonotonicity and equivalent relations with other classes, however, provide convergence analysis of a hybrid method of proximal point and extragradient for cohypomonotone \ref{eq:vi}s.
\citet{Alber2004} rely on the \emph{strong hypomonotonicity} structure, a subset class of hypomonotonicity, to show that a Tikhonov-Browder operator regularization is stable for data perturbations yielding operators that satisfy this assumption.
In this work, we concentrate on the relatively unexplored class of hypomonotone problems, with the goal of gaining deeper insights into its instances and investigating whether the standard extragradient method achieves convergence within this class.

\section{Preliminaries}\label{sec:preliminaries}

This section presents the method under consideration and the essential definitions needed for the main results, with further relevant definitions included in Appendix~\ref{app:background}.

\noindent\textbf{Notation.}
We denote 
\begin{enumerate*}[series = tobecont, itemjoin = \enspace, label=(\roman*),font=\itshape]
\item vectors with small bold letters,
\item real-valued functions with small letters,
\item operators with capital letters, e.g., $\F$ or $\T$, 
\item matrices with bold capital letters, and
\item sets with curly capital letters.
\end{enumerate*}
$\mathbf{A}^\dagger$ and $\mA^\ddagger$ denote the complex conjugate and conjugate transpose of the matrix $\mA$, respectively.
For a complex number $c \in \CC$ we write
$c = \mathfrak{R}(c) + i \mathfrak{I} (c)$ where $\mathfrak{R}(c)$ is the real
and $\mathfrak{I}(c)$ is the imaginary part of $c$.

\smallskip
We study the extragradient algorithm defined by the following update rule.

\noindent\textbf{Extragradient~\citep{korpelevich1976extragradient}.} At iteration $k$, to the current iterate $\vx_k$ it applies the operator at an iterate $\vy_{k+1}$ that is extrapolated using gradient descent, as follows:
\begin{equation}\tag{EG} \label{eq:eg}
\begin{aligned}
\vy_{k+1} &= \vx_k - \gamma \F(\vx_k) \\
\vx_{k+1} &= \vx_k - \gamma \F(\vy_{k+1}) 
\end{aligned} \,,
\end{equation}
where $\gamma > 0$ denotes the step size.

\medskip
Monotonicity is defined as follows.
\begin{definition}[Monotonicity]\label{def:monotone}
An operator \( \F \colon \R^n \to \R^n \) is said to be \emph{monotone} iff:
\begin{equation}\label{eq:monotone}\tag{Mnt}
\langle \F(\vx) - \F(\vy), \vx - \vy \rangle \geq 0 \,, 
\qquad \forall \vx, \vy \in \R^n.
\end{equation}
\end{definition}

\medskip
The following definition introduces a relaxed form of monotonicity where an operator’s inner product with the difference of its arguments is bounded below by a negative value.
\begin{definition}[Hypomonotonicity]\label{def:hypomonotone}
An operator \( \F \colon \R^n \to \R^n \) is \emph{hypomonotone} with modulus \( \mu \geq 0 \) iff:
\begin{equation}\label{eq:hypomonotone}\tag{HM}
\langle \F(\vx) - \F(\vy)\,, \vx - \vy \rangle \geq -\mu \norm{\vx - \vy }^2, \qquad \forall \vx, \vy \in \R^n \,.
\end{equation}
\end{definition}

\medskip
The following property ensures bounded changes in \( \F \) with respect to changes in its input, which is essential for proving convergence results in iterative algorithms.
\begin{definition}[Lipschitz operator]\label{def:lip}
An operator \( \F \colon \R^n \to \R^n \) is said to be $L$-\emph{Lipschitz continuous} with constant \( L > 0\) if:
\begin{equation}\label{eq:lip}\tag{Lip}
    \norm{ \F(\vx) - \F(\vy) } \leq L \norm{ \vx - \vy }, \qquad \forall \vx, \vy \in \R^n \,.
\end{equation}
\end{definition}

\medskip
Normal matrices play a crucial role in diagonalization and spectral analysis.
\begin{definition}[Normal Matrix] \label{def:normal}
    A matrix \( \mathbf{A} \in \CC^{n \times n} \) is \emph{normal} if it commutes with its conjugate transpose, that is,
    $
        \mathbf{A} \mathbf{A}^\dagger = \mathbf{A}^\dagger \mathbf{A} \,.
    $
\end{definition}

\medskip
For normal operators $\T$, the following spectral theorem guarantees a decomposition into orthogonal eigenvectors and eigenvalues, simplifying the analysis of $\T$.

\begin{theorem}[Spectral Theorem]\label{thm:spectral}
    Any normal matrix \( \mathbf{A} \)---as per Def.~\ref{def:normal}---can be diagonalized by a unitary matrix $\mathbf{S}$, yielding:
        $ \mathbf{A} = \mathbf{S}\mathbf{D} \mathbf{S}^\dagger \,,$
where $\mathbf{D}$ is a diagonal matrix of eigenvalues. 
\end{theorem}
\begin{lemm}[Normality of Matrix Representation]\label{lem:normal-operator}
  Let $\T \colon \H \rightarrow \H$ be a normal operator on a finite-dimensional Hilbert space $\H$. 
  Consider an orthonormal basis $\mathcal{B}=\{ \ve_1,\dots, \ve_n\}$ of $\H$. Then the matrix representation of $\T$ with respect to $\mathcal{B}$ is a normal matrix, i.e., it satisfies 
  $\mM \mM^\ddagger = \mM^\ddagger\mM$, where $\mM$ is the matrix representation of $\T$ with respect to $\mathcal{B}$ and $\mM^\ddagger$ is its conjugate transpose.
\end{lemm}

\section{Hypomonotone Operators: Examples \& Subclass Characterization}\label{sec:hm_examples}

In this section, we first give two simple examples of hypomonotone operators that are not monotone.
Following that, we characterize a parametric hypomonotone subclass that is not monotone.

\paragraph{Example 1: concave problem.}
Consider the following operator $\T_1 \colon \R^2  \to \R^2$, defined as:
$$ 
\T_1 ( \vx ) \triangleq -\mu \vx,  \qquad \text{with} \quad \mu \geq 0 \,.
$$ 
To show non-monotonicity, consider the following two points 
$ \vx=[1,0]^\intercal, \vy=[0,1]^\intercal  \in \R^2$. 
We have,  $ \langle \T_1(\vx) - \T_1(\vy),\vx - \vy \rangle = -2 \mu < 0\,.$ 

To show that $\T_1$ is hypomonotone, replacing for the LHS of \eqref{eq:hypomonotone} gives: 
$$
\langle -\mu \vx + \mu \vy, \vx - \vy \rangle =
-\mu \langle \vx - \vy, \vx - \vy \rangle \geq
-\mu \| \vx - \vy\|^2 \,.
$$ 
Since $\mu >0 $, $\T_1$ is $\mu$-hypomonotone. 

Also notice that, $\|\T_1(\vx) - \T_1(\vy)\| = |-\mu|\|  \vx -  \vy \| \leq \mu\|  \vx -  \vy \|,$ hence $\T_1$ is $\mu$-Lipschitz.

\paragraph{Example 2: affine problem.}
Consider the operator $\T_2\colon \R^2 \to \R^2 $ defined as follows:
$$
    \T_2(
    \begin{bmatrix}
    x \\
    y
    \end{bmatrix}
   ) \triangleq 
   \begin{bmatrix}
    2 x + y - 1 \\
    -x - 1.5 y + 1
    \end{bmatrix} \,.
$$ 

To show that $\T_2$ is not monotone, 
consider the vectors 
$\vz_1 = [0,0]^\intercal, \vz_2 = [0,1]^\intercal$. 
We have $\langle \T_2(\vz_1) - \T_2(\vz_2), \vz_1 - \vz_2 \rangle = -\frac{3}{2} < 0 \,$ or easier we can verify that \(T_2\)  is hypomonotone and not monotone by calculating the eigenvalues of $\T_2$ and using ~\ref{th:spe-charac-normal-operator}, in particular $\lambda=\frac{1}{4}(1-\sqrt{33})<0$.

\subsection{Hypomonotone Subclass: Operators with Negative Eigenvalues}\label{sec:}

\begin{theorem}[Normal Spectral Characterization Theorem]\label{th:spe-charac-normal-operator}
Let $\mathcal{H}$ be a finite-dimensional Hilbert space over $\R$ and let $\T\colon \H \to \H$ be a $normal$ linear operator with at least one eigenvalue $\lambda_i \in \text{Spec}(\T)$ whose real part is negative $\mathfrak{R} (\lambda_i) < 0$. Then, the operator $\T$ is hypomonotone and not monotone.
\end{theorem}

\begin{proof}{\textbf{of Theorem~\ref{th:spe-charac-normal-operator}: Normal Spectral Characterization Theorem.}}

We first show that $\T$ is hypomonotone.
 
\paragraph{Hypomonotonicity of $\T$.}
   Let $n$ be the dimension of $\H$ and let $\ve_1, \dots ,\ve_n$ be the standard basis of $\R^n$. Let $\mA$ be the matrix representation of $\T$ with respect to the standard basis. By  Lemma~\nameref{lem:normal-operator}, it follows that $\mA$ is normal. By the \nameref{thm:spectral}, there exists a unitary matrix $\mathbf{S}$ and complex numbers $\lambda_1,\dots,\lambda_n$ such that $\mathbf{A}=\mathbf{S}\,\mathbf{\operatorname{diag}}(\lambda_1, \ldots, \lambda_n)\,\mathbf{S}^{-1}$, that is $\mathbf{\operatorname{diag}}(\lambda_1, \ldots, \lambda_n)=\mathbf{S}^{-1}\,\mathbf{A}\,\mathbf{S}$.
   
    Consider a point $\vx\in \H$, with respect to the new basis representation, write \( \vx = \sum_{k=1}^{n} c_k v_k\) with $c_k \in \mathbb{C}$ and $\vv_1, \dots , \vv_n$ are eigenvectors that form the orthonormal basis. We show next that for every $ \vx, \vy \in \H$  the inequality~\eqref{eq:hypomonotone} is valid.

    \noindent
    Since $\langle \cdot,\cdot \rangle \in \R$, then
    \[
    \langle \T(\vx) - \T(\vy), \vx - \vy \rangle  = \mathfrak{R} \Big(
        \langle \T(\vx), \vx \rangle 
        - \langle \T(\vx), \vy \rangle 
        - \langle \T(\vy), \vx \rangle 
        + \langle \T(\vy), \vy \rangle \Big) \,.
  \]

Let $ \vx=\sum_{k=1}^{n} a_{k} \vv_{k}\,,$ and 
    $\vy=\sum_{l=1}^{n} b_{l} \vv_{l}$. 
    Since $\langle \vv_i, \vv_j \rangle = 0$ for $i\neq j$ it follows: 
    \[
    \langle \T(\vx), \vy \rangle =\sum_{k=1}^{n}\lambda_ka_{k}\overline{b_{k}}\,, \qquad\langle \T(\vy), \vx \rangle =\sum_{k=1}^{n}\lambda_k\overline{a_{k}}b_{k}\,,
    \] 
    \[
    \langle \T(\vx), \vx \rangle =\sum_{k=1}^{n}\lambda_k|a_{k}|^2,\qquad\langle \T(\vy),\vy \rangle =\sum_{k=1}^{n}\lambda_k|b_{k}|^2.
    \] 
    Now it follows that 
    \[ 
    \mathfrak{R} \big( \langle \T(\vx)-\T(\vy),\vx-\vy \rangle \big) = \mathfrak{R} \big( \sum_{k=1}^{n}\lambda_k (|a_k|^2+|b_k|^2 - (a_k\overline{b_k}+\overline{a_k}b_k)) \big)
    \] 
    \[
    =\mathfrak{R}\big( \sum_{k=1}^{n} (\lambda_k|a_k|^2+|b_k|^2)\big)-\mathfrak{R}\big( \sum_{k=1}^{n}(2\lambda_k\operatorname{Re}(a_kb_k)) \big)
    \] 
    \[
    \geq \mathfrak{R}\big(\sum_{k=1}^{n} (\lambda_k|a_k|^2+|b_k|^2) \big)-\mathfrak{R}\big( \sum_{k=1}^{n}(2\lambda_k|a_kb_k|) \big)
    \] 
     \[
     = \mathfrak{R}\big(\sum_{k=1}^{n} (\lambda_k(|a_k|^2+|b_k|^2-2|a_kb_k|)) \big) 
     \] 
     \[
     \geq \mathfrak{R}\big(\sum_{k=1}^{n} \lambda_k|a_k-b_k|^2\big) \geq \mathfrak{R} (\lambda_{min} ) \| \vx-\vy\|^2.
     \] 
By the assumption of the theorem, there exists an eigenvalue \(\lambda_{\min}\) with a negative real part, i.e., \(\mathfrak{R}(\lambda_{\min}) < 0\).  
Let \(\mu = -\mathfrak{R}(\lambda_{\min})\). This establishes the result in equation~\eqref{eq:hypomonotone}.

\paragraph{Non-monotonicity.}
It remains to show that $\T$ is not monotone.
By assumption, \(\T\) has at least one eigenvalue \(\lambda \in \mathbb{C}\) whose real part is negative, say
\[
  \lambda \;=\; a + i\,b
  \quad\text{with}\quad
  a = \mathfrak{R} (\lambda) < 0.
\]

Let \(\vv \in \mathbb{C}^n\) be an eigenvector corresponding to \(\lambda\).  Write
\[
  \vv \;=\; \vp \;+\; i\, \vq,
  \quad
  \text{where } \vp, \vq \in \R^n.
\]
Notice that the real span of \(\{\vp, \vq\}\subset \R^n\) is invariant under \(\T\).  Indeed, in the basis \(\{ \vp, \vq\}\) of this 2D real subspace, the matrix representation of \( \T\) is
\[
  \begin{bmatrix}
    a & -b \\
    b & a
  \end{bmatrix},
\]
which acts on any \(\, \vx = x_1\, \vp + x_2\, \vq \in \mathrm{span}\{\vp, \vq\}\subset \R^n\) as a rotation-and-scaling by \((a,b)\).

\smallskip

We now claim that there is a nonzero real vector \( \vx \) for which it holds:
\[
  \langle \T(\vx),\, \vx\rangle \;<\;0,
\]
violating monotonicity.  
Indeed, take \( \vx = x_1\, \vp + x_2\, \vq\).  A straightforward calculation in this 2D subspace shows
\[
  \langle \T(\vx),\, \vx\rangle 
  \;=\; \big\langle\, \begin{bmatrix}a & -b \\ b & a\end{bmatrix}
  \begin{bmatrix}x_1 \\ x_2\end{bmatrix},\;\begin{bmatrix}x_1 \\ x_2\end{bmatrix}\big\rangle_{\!\R^2}
  \;=\; a\,(x_1^2 + x_2^2).
\]
Since \(a<0\), we immediately get 
\[
  \langle \T(\vx),\, \vx\rangle \;<\;0
  \quad\text{for any nonzero } \vx\in \mathrm{span}\{ \vp, \vq\}.
\]
Hence for \(\vy=\mathbf{0}\), 
\[
  \langle \T(\vx)-\T(\vy),\, \vx-\vy\rangle 
  \;=\; \langle \T(\vx),\, \vx\rangle
  \;<\;0,
\]
which contradicts the monotonicity definition.  Therefore, \(T\) is \emph{not} monotone.

\end{proof}

\section{Divergence of Extragradient on Hypomonotone VIs}\label{sec:eg_divergence}

In this section, we show that the extragradient method does not converge for a general instance of the hypomonotone class of VI problems. In particular, we provide a simple counter-example.

\begin{theorem}[\ref{eq:eg} divergence on \ref{eq:hypomonotone}.]\label{th:divergence}
    Let $\gamma>0$ be the step size for the \ref{eq:eg} method.
    Let $ \vx_0 = [1,0]^\intercal$ be the initial point for the \ref{eq:eg} method, and let $\F(\vx)=\mA \vx$ to be an operator such that: 
\[
\mA = 
\begin{bmatrix} -2 & 0 \\ -1 & -2 \end{bmatrix} \,.
\]
Then, the sequence  $\|{\vx_n}\|$ generated by \ref{eq:eg} diverges as $n\rightarrow \infty.$
\end{theorem}

\begin{proof}
    By the extragradient \eqref{eq:eg} update rule, we have $\vy_{k+1}= \vx_k- \gamma \mA \vx_k$ and $\vx_{k+1}=\vx_k- \gamma \mA \vy_{k+1}$.
    Hence, $\vx_{k+1}=\vx_k(\mathbf{I}-\gamma \mA + \gamma^2\mA^2)$.

    Define $\mM (\gamma)=\mI -\gamma \mA+\gamma^2\mA^2$.  
    By the definition of $\mA$, we can see that: 
   $$
     \mM(\gamma) = \begin{bmatrix}
    1+2\gamma+4\gamma^2 & 0\\
    \gamma+4\gamma^2     & 1+2\gamma+4\gamma^2
    \end{bmatrix}.
   $$
Now consider $\vx_0$ as defined above and $\vx_1=\mathbf{M}(\gamma) \vx_0 =(1+2\gamma+4\gamma^2,\gamma+4\gamma^2)$.

Let $\|\cdot\|$ to be the euclidean norm and for all $\gamma$ such that $0<\gamma <1$: \\
$$ \|\vx_1\|^2=  (1+2\gamma+4\gamma^2)^2 + (\gamma+4\gamma^2)^2 >1\,.$$

Calculating $\vx_2$, we have:
\[
\vx_2 = \mathbf{M}(\gamma) \vx_1 = 
\begin{bmatrix}
(1+2\gamma+4\gamma^2)^2 \\ 
2(\gamma+4\gamma^2)(1+2\gamma+4\gamma^2)
\end{bmatrix},
\]
and  $\|\vx_2\| >\|\vx_1\|$. 
Furthermore iteratively 
\[
\vx_k = \mathbf{M}(\gamma)^k \vx_1 = 
\begin{bmatrix}
(1+2\gamma+4\gamma^2)^k \\ 
k (\gamma+4\gamma^2)(1+2\gamma+4\gamma^2)^{k-1}
\end{bmatrix}
\]
Thus, $\|\vx_k\| >\|\vx_{k-1}\|$; and hence $\|\vx_k\| \rightarrow \infty.$
\end{proof}

\begin{remark}
     Notice that the operator $\F$ is hypomonotone but not monotone. This follows immediately by Theorem ~\ref{th:spe-charac-normal-operator}, since we have a negative eigenvalue; in particular, $\lambda=-2$.
\end{remark}

\subsection{Numerical Experiment}

In this section, we consider the following game:
\begin{equation}\label{eq:numerical_exp}\tag{Exp-Game}
    \min_x \max_y \ \    -x^{2} + x y + y^{2} \,,
\end{equation}
with $x,y \in\R$.
The corresponding operator of \eqref{eq:numerical_exp} is:
\begin{align*}
\F( \begin{bmatrix} x\\ y\end{bmatrix}) = \begin{bmatrix}
-2x + y \\
- x - 2y
\end{bmatrix} \,.
\end{align*}
Notice that the eigenvalues of the operator above have a negative real part, in particular, $\mathfrak{R}(\lambda)=-2$. Hence, by Theorem~\ref{th:spe-charac-normal-operator} it follows that \eqref{eq:numerical_exp} is hypomonotone and not monotone.

\begin{figure}[th]
\centering
\includegraphics[width=.55\linewidth,trim={0cm .1cm .1cm .1cm},clip]{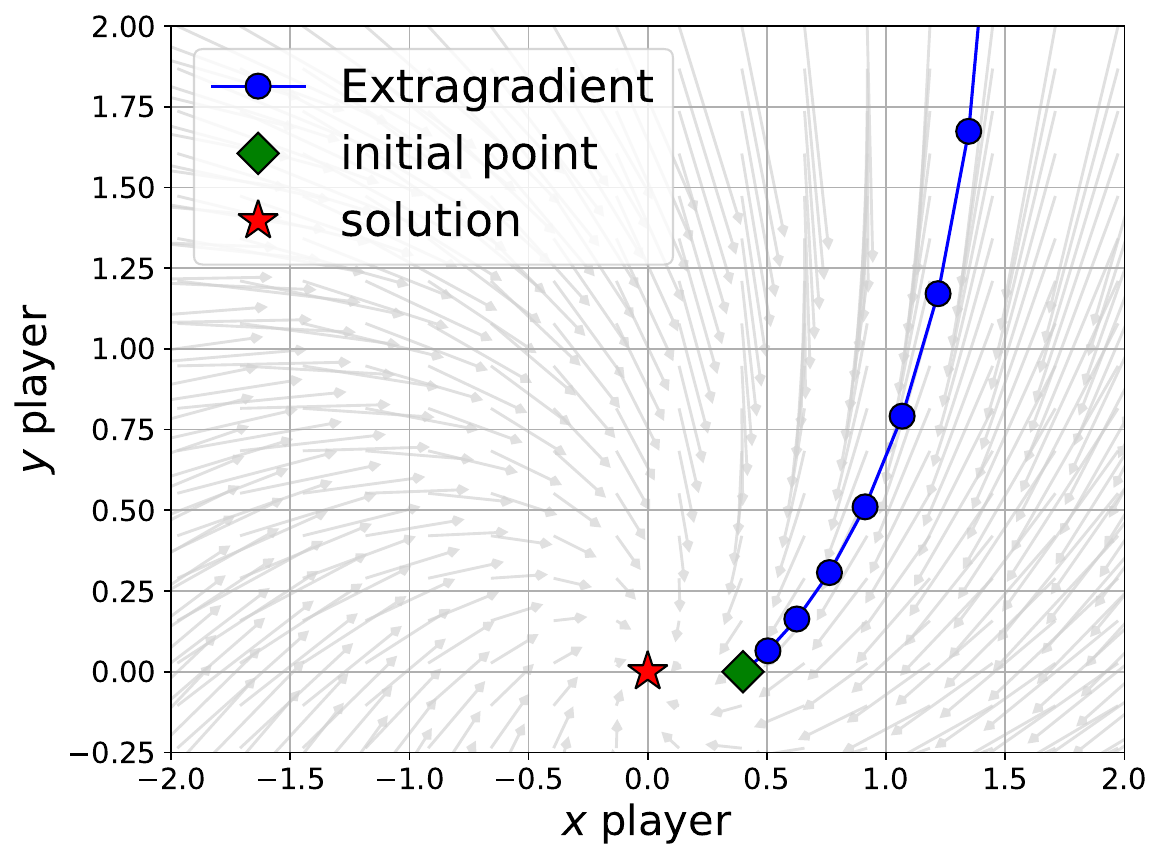}
\caption{
Trajectory of the \eqref{eq:eg} method (blue) on the \eqref{eq:numerical_exp} problem; with $\gamma=0.5$. 
The gray arrows depict the vector field of the game.
The \ref{eq:eg} method diverges for this example.
}
\label{fig:num_exp}
\end{figure}

Figure~\ref{fig:num_exp} depicts the trajectory of the \ref{eq:eg} method on the \eqref{eq:numerical_exp} problem, for a fixed step size. We observe that \ref{eq:eg}  diverges for this problem. We note that \ref{eq:eg} also diverges when we select smaller step sizes.

\section{Discussion}

In this paper, we focus on a specific subclass of hypomonotone problems by considering operators that are both normal and linear. For this restricted subclass, we provide a detailed characterization based on the eigenvalues of the associated operators.
We demonstrate the divergence of the extragradient method in the hypomonotone setting by constructing a counterexample. In this example, for a given initial point, the generated sequence diverges in norm regardless of the chosen step size. This highlights the pressing need for a deeper understanding of hypomonotone problems and the development of methods that can ensure convergence for problems that are hypomonotone but not monotone.

\bibliography{main}

\clearpage
\appendix
\section{Additional Background}\label{app:background}

This section provides additional definitions omitted from the main text.

\bigskip
For convenience, we define the graph of an operator as follows.

\begin{definition}[Operator Graph]\label{def:graph}
Given an operator $\F\colon \mathcal{X} \to \mathcal{Y}$, its \emph{graph} $\text{Graph} \F$ is:
\begin{equation}\tag{Graph}\label{eq:graph}
\text{Graph} \F \triangleq 
\{ (\vx, \vu ) | \vu \in F(\vx) , \forall \vx \in \text{dom} \F \} \subseteq \mathcal{X}  \times \mathcal{Y} \,.
\end{equation}
\end{definition}

\begin{definition}[Inverse Operator]\label{def:invese-operator}
Let $\F\colon \mathcal{X} \to \mathcal{Y}$.
The \emph{inverse} of the operator $\F$, that is $\F^{-1}$, is then:
\begin{equation}\tag{Inv}\label{eq:inv}
    \F^{-1} \triangleq \{ (\vu, \vx) | (\vx, \vu) \in \text{Graph}F\} \,.
\end{equation}
\end{definition}

The above definition also covers multi-valued operators.
Notice that the operator inverse is not a classical definition. That is, unlike for the inverse function, the operator inverse is always defined, and in the general case, it is possible $\F^{-1}\F \neq \mI$, where $\mI$ is the identity operator $\mI \triangleq \{ (\vx, \vx) | \vx \in \R^n \}$.

\bigskip
\emph{Cohypomonotonicity} is a related class to hypomonotonicity, in that it establishes a hypomonotonicity relationship but with the \emph{inverse} of the operator.
\begin{definition}[Cohypomonotonicity]\label{def:cohypomonotone}
An operator $ \F\colon \mathcal{X}\to \mathcal{Y} $ is cohypomonotone if its inverse $F^{-1}$ is hypomonotone with respect to a constant \( \rho \geq 0 \) on the set \( F(C) \), that is, 
\begin{equation} \tag{cHM} \label{eq:cohypomonotone}
    \langle \F^{-1}(\vx) - \F^{-1}(\vy), \vx - \vy \rangle \geq -\rho \| \vx - \vy \|^2  \,, \qquad \forall \vx, \vy \in \mathcal{X} \,.
\end{equation}
\end{definition}

\bigskip
The following gives the proof of Lemma~\ref{lem:normal-operator}.
\begin{proof}{of Lemma~\ref{lem:normal-operator}.}
   Let $\mathbf{M}$ be the matrix representation of $\operatorname{T}$ with respect to the standard orthonormal basis---that is, $\mM_{i,j}=\langle \operatorname{T}(e_j),e_i) \rangle$. 
   Then,
   $$
  \overline{ \langle  \operatorname{T}(e_j),e_i \rangle} = \langle e_i,\operatorname{T}(ej) \rangle = \langle \operatorname{T}^\dagger(e_j), e_i \rangle = \mathbf{M}^\dagger_{i,j} \,.
  $$
\end{proof}

\end{document}